\documentclass[preprint,12pt]{elsarticle}
\usepackage{amsmath,amsfonts,amsthm,amssymb,amsbsy}
\usepackage{subfig}
\newtheorem{thm}{Theorem}[section]
\theoremstyle{definition}

\theoremstyle{plain}
\newtheorem{lem}[thm]{Lemma}
\DeclareMathOperator*{\epi}{\Pi}
\usepackage{mathtools}

\usepackage{lineno}

\journal{Computers and Mathematics with Applications}

\begin{document}

\begin{frontmatter}



\title{ A Unified Analysis of Nonconforming Virtual Element Methods for Convection Diffusion Reaction Problem }


\author{Dibyendu Adak, E.Natarajan}

\address{Department of Mathematics, \\
Indian Institute of Space Science and Technology, \\
Thiruvananthapuram-695547, Kerala, India }

\begin{abstract}
We discuss nonconforming virtual element method for convection dominated (diffusive coefficient is very small compared to convective coefficient and reaction coefficient ) convection-diffusion-reaction equation using $L^{2}$ projection operator.In this paper we stabilize the stabilization terms $\int_{T}(\vec{b}.\nabla u)(\vec{b}.\nabla v)$ using same technique which is used for  stabilization of symmetric part in VEM, where $T$ is an  arbitrary element, and assume $H^{2}(T)$ regularity of $v_{h}|_{T}$ on each element to prove polynomial consistency where $v_{h}$ is approximate solution.We have shown that linear nonconforming VE is not convergent for convection dominated convection-diffusion reaction problem and higher regularity of $f$,source term is also needed for convergence analysis.The novelty of this paper is we introduce a new SDFEM type nonconforming virtual element method for convection-dominated convection diffusion reaction equation, and discuss the computability issue using degrees of freedom of element without   explicit knowledge of basis functions of virtual element methods.The present framework is stable in the limit of vanishing diffusion.
\end{abstract}


\end{frontmatter}


\section{Introduction}
\label{temp1}
In recent times the virtual element method has been successfully analysed and applied to a 
great variety of problems in $2$D as well as $3$D. Virtual element method has a noticeable 
similarity with mimetic finite difference method 
\cite{brezzi2009mimetic, de2014nonconforming,da2008higher,droniou2010unified}. The advantage 
of virtual element method is that we can generalize analysis framework for any order of 
desired accuracy and we can compute solution on each element $T$ using only degree of freedoms 
without having explicit knowledge of basis functions. The test and trial functions contain 
polynomials of degree $\leq$ $k$ with some additional continuous functions which are 
solution of model problem.

Conforming virtual element has been studied for elliptic equation in \cite{beirao2013basic}, linear 
elasticity problem in \cite{da2013virtual}, biharmonic plate bending problem 
\cite{brezzi2013virtual}. Recently nonconforming virtual element method has been analysed for 
diffusion problem by Ayuso et al \cite{de2014nonconforming}. The variational 
or weak formulation of all these above said model problem are symmetric hence stability 
analysis and polynomial consistency can be easily carried out using elliptic projector 
operator\cite{ahmad2013equivalent}, for any order of accuracy $k$. But if the variational 
or weak formulation is non symmetric then the stability analysis does not follow easily and 
requires further analysis.

In this paragraph we discuss about the computational aspects of the methods to be proposed. We compute moments of 
test function using degrees of freedom of finite element. It has been shown in 
\cite{beirao2013basic,ahmad2013equivalent} that elliptic projector operator is computable 
for polynomial of degree upto $k$ where $k$ is accuracy of the element, but $L^{2}$ 
projection operator is partially computable when degree of polynomial $\leq$ $k-2$. But 
after little modification of virtual element space we can compute $L^{2}$-projection using 
elliptic projector operator which is explained in \cite{ahmad2013equivalent}. The main 
advantage of using $L^{2}$ projection is we can use this operator even for non-symmetric 
bilinear form like convection-diffusion-reaction equation, and it has been observed that 
polynomial consistency, stability analysis and convergence analysis can be easily carried 
out without any difficulty.

In this paper we consider the convection-diffusion-reaction problem
\begin{eqnarray}
-\epsilon \, \Delta u+\vec{b} \cdot \nabla u+c \, u & = & f \,\,\,  \text{in} \, \,  \Omega \nonumber \\ 
                                     u & = & u_{b} \,\,  \text{on} \, \, \partial \Omega 
                                     \label{temp2}
\end{eqnarray}
under the assumptions that $\vec{b},c,f$ are sufficiently smooth functions and $\epsilon$ is a 
small perturbation $(\epsilon \ll 1)$. The above said problem is not stable in standard 
Galerkin method hence numerical solution produce non-physical oscillation at boundary 
layer, which contaminates the global computed numerical solution. There are several remedy 
to overcome this oscillation. Streamline-diffusion finite element is famous among 
them. Discrete bilinear form $a_{h}(u_{h},v_{h})$ is coercive 
i.e. $a_{h}(v_{h},v_{h})\geq C \| v_{h} \|^{2}_{1}$ in $H^{1}(\Omega)$ norm which is 
unable to capture large growth of $\|\vec{b} \cdot \nabla u \|_{0}$ at the boundary layer 
hence produce non physical oscillation. In streamline diffusion method we add additional 
term $\int_{T}|\vec{b} \cdot \nabla u|^{2} \ dT$ and introduce new convective term dependent 
norm $|||.|||$ to capture this oscillation. Using the domain dependent norm a 
new nonconforming analysis of streamline diffusion method for convection dominated 
convection-diffusion equation has been studied by 
Tobiska et al \cite{knobloch2003p,roos2008robust,john1997nonconforming}.

In this paper we analyse non conforming virtual element method for convection-diffusion-reaction 
equation using external $L^{2}$ projection operator. To prove polynomial consistency 
we have assumed higher regularity $u_{h}|_{T} \in H^{3}(T)$ on each element $T$. In 
nonconforming VEM approximate solutions are discontinuous along interior edges except 
at Gauss-Lobatto points on the edges. In convergence analysis we see that two additional 
jump term which is known as consistency 
error \cite{knobloch2003p,brenner2008mathematical,ciarlet2002finite}, arise  because of this 
nonconformity along interior edges. We shall use patch 
test \cite{irons1972experience,fortin1983non} which is a crucial property of nonconforming 
method to bound these jump terms. The layout of the paper is as follows. In 
section \ref{sec1} we discuss continuous weak formulation and basic setting of the model 
problem $(\ref{temp2})$. Since construction of non conforming virtual element and degrees of freedom 
are already defined in \cite{de2014nonconforming} we introduce those discussion with local 
and global settings of nonconforming VEM  briefly in section \ref{sec3}. Discrete stability 
and polynomial consistency have been studied in section \ref{sec7}. Using mesh 
dependent norm \cite{john1997nonconforming} discrete coercivity has been discussed 
in section \ref{sec9}. Boundedness of discrete and continuous bilinear form have been 
studied in section \ref{sec10}. Consistency error estimation is discussed in  
section \ref{sec12} and computational issue is discussed in 
section \ref{sec13}. In the section \ref{sec14} we discuss convergence analysis 
and a priori error estimation. Finally some conclusions are drawn in section \ref{sec16}. 

\section{Continuous Problem \label{sec1}} 
Let the domain $\Omega$ in $\mathbb{R}^{d}$ with 
$d=2,3$ be a bounded open polygonal domain with straight boundary edges for $d=2$ and 
polyhedral domain with flat boundary faces for $d=3$. We consider the model problem 
described in (\ref{temp2}) with sufficient regularity of $\vec{b},c,f$. The 
variational formulation of the model problem (\ref{temp2}) reads as, find $u\in V$ such that 
\begin{equation}
 A(u,v)=<f,v> \quad \forall \ v \in V
\label{temp3}
\end{equation}   
where the bilinear form $A(\cdot,\cdot): V \times V \rightarrow \mathbb{R}$ is given by 

\begin{equation}
A(u,v)=\int_{\Omega} \epsilon \, \nabla u \cdot \nabla v \ dT +\int_{\Omega} \vec{b} \cdot \nabla u\ v+\int_{\Omega} c \, u \, v  \quad \forall \ u,v \in V
\nonumber
\end{equation}
and \begin{equation}
<f,v>=\int_{\Omega} f v \nonumber
\end{equation}
$<\cdot,\cdot>$ denotes the duality product between the function space 
$V^{'}$ and $V$, where $V$ is defined by $V:=H^{1}_{0}(\Omega)$.

We consider the assumption
\begin{equation}
\left(c-\frac{1}{2} \, \nabla \cdot \vec{b} \right) \geq c_{0}>0
\label{asmp1}
\end{equation}

We define the elemental contributions of the bilinear form $A(\cdot,\cdot)$ by
\begin{eqnarray}
a(u,v) &:=&\int_{\Omega} \epsilon \, \nabla u \cdot \nabla v \ d\Omega \nonumber \\
b(u,v) &:=&\int_{\Omega} (\vec{b} \cdot \nabla u) \, v \ d\Omega \nonumber \\
c(u,v) &:=& \int_{\Omega} c \, u \, v \, d\Omega \nonumber
\end{eqnarray}

\begin{eqnarray}
A(v,v) &=& \int_{\Omega} \epsilon \, \vert \nabla v \vert^{2}+\int_{\Omega}(\vec{b} \cdot \nabla v) \, v+\int_{\Omega} c \, v^{2} \nonumber \\
       &=& \int_{\Omega} \epsilon \, \vert \nabla v \vert^{2} +\int_{\Omega} \left(c-\frac{1}{2} \nabla \cdot \vec{b} \right) v^{2} \nonumber \\
       &\geq & \int_{\Omega} \epsilon \, \vert \nabla v \vert^{2}+\int_{\Omega} c_{0} \, v^{2} \nonumber \\
       &\geq & C \, \|v\|^{2}_{H^{1}(\Omega)}
       \label{temp4} 
\end{eqnarray}

Inequality (\ref{temp4}) implies that $A(\cdot,\cdot)$ is coercive in 
$H^{1}(\Omega)$ norm. Together with coercivity and boundedness of 
$A(\cdot,\cdot)$ Lax-milgram theorem ensures that there exists a unique solution to 
the variational form (\ref{temp3}).
 
\subsection{Basic Setting \label{sec2}} 
Let $\{ \tau_{h}\}_{h}$ be a family of triangulation of the domain $\Omega$. Each 
triangulation $\tau_{h}$ consists of a finite number of elements $T$ such 
that $\bar{\Omega}= \underset{T\in \tau_{h}} \bigcup T$ and any two different 
elements $T_{1},T_{2} \in \tau_{h}$ are either disjoint or possess 
either a common vertex or a common edge. We assume following on the family of partitions.
 
There exists a positive $\rho >0 $ \ such that 
\begin{itemize} 
\item[(Z1)] for every element $ T$ and for every edge/face $e\subset \partial T $, we have $h_{e}\geq \rho h_{T} $.
\item[(Z2)] every element $T $ is star-shaped with respect to all the points of a sphere of radius $\geq \rho h_{T} $.
\item[(Z3)] for $d=3$, every face $e \in \varepsilon_{h}$ is star-shaped w.r.t. all the points of a disk having radius   $\geq \rho h_{e} $.
\end{itemize}
 
$\varepsilon^{0}_{h}$ and $\varepsilon^{\partial}_{h}$ denote the set of interior and 
boundary edges respectively. The maximum of the diameters of element $T \in \tau_{h}$ 
will be denoted by $h$, 
\begin{center}
 i.e.\quad$ h=\underset{T \in \tau_{h}}{\mathrm{max}}\{h_{T}\} $
 \end{center}
 We introduce the broken sobolev space for any integer $m>0$
 \begin{equation}
 H^{m}(\tau_{h}):=\epi_{T\in \tau_{h}}H^{m}(T)=\{v \in L^{2}(\Omega):v|_{T} \in H^{m}(T)\} \nonumber
 \end{equation}
 and define the broken $H^{m}$ norm
 
 \begin{equation}
 \|v\|^{2}_{m,h}=\sum_{T \in \tau_{h}} \|v \|^{2}_{m,T} \quad \forall v\in H^{m}(\tau_{h}) \nonumber
 \end{equation}
 For $m=1$ the broken $H^{1}$-seminorm 
 \begin{equation}
 \vert v \vert^{2}_{1,h}=\sum_{T\in \tau_{h}} \|\nabla v\|^{2}_{0,T} \quad \forall v \in H^{1}(\tau_{h}) \nonumber
 \end{equation}
 
 Let $e\in \varepsilon^{h}_{0}$ be  an interior edge and let $T^{+}$ and $T^{-}$ be two 
 triangle that shares the common edge $e$ and let $n^{\pm}_{e}$ denote the unit normal on $e$ in the outward 
 direction w.r.t. $T^{\pm}$. We then define the jump operator as:
 \begin{equation}
 [|v|]:=v^{+}n^{+}_{e}+v^{-}n^{-}_{e} \quad \text{on} \quad e\in \varepsilon^{0}_{h} \nonumber
 \end{equation}
 and 
 \begin{equation}
 [|v|]:= vn_{e} \quad \text{on} \quad e\in \varepsilon^{\partial}_{h} \nonumber
 \end{equation}

Now we introduce the space that satisfies the continuity at Gauss-Lobatto points on the edges. For 
an integer $k\geq 1$, we define 
\begin{equation}
H^{1,nc}(\tau_{h};k)=\left\{v\in H^{1}(\tau_{h}):\int_{e}[|v|] \cdot n_{e} \, q \, ds=0 \quad q\in \mathbb{P}^{k-1}(e), \forall e\in \varepsilon_{h} \right\} \nonumber
\end{equation}
$H^{1,nc}(\tau_{h};k)$ satisfy patch-test of order $k$ 
\cite{irons1972experience}. $H^{1,nc}(\tau_{h},1)$ is the space with minimal required 
order of patch test to ensure convergence analysis of diffusion dominated 
convection-diffusion-reaction equation. But if the problem is convection dominated then 
we need higher order of patch test atleast $k=2$ which is explained in section \ref{sec13}.

\section{Nonconforming Virual Element Method \label{sec3}}
We have already discussed the continuous setting of variational form of 
model problem (\ref{temp2}) in section \ref{sec1}. We shall present the variational 
form in a different way which will help us to prove the convergence analysis without 
disturbing the weak solution and its corresponding weak formulation (\ref{temp3}).

Applying integration by parts to the convective term $(\vec{b} \cdot \nabla u,v)$ we obtain
\begin{equation}
\int_{\Omega}(\vec{b} \cdot \nabla u) \, v = \frac{1}{2} \left[\int_{\Omega}(\vec{b} \cdot \nabla u)v-\int_{\Omega}(\vec{b} \cdot \nabla v)u+\int_{\Omega}\nabla \cdot \vec{b} \, u \, v \right] \quad \forall\  u,v \in H^{1}_{0}(\Omega) \nonumber
\end{equation}
Bilinear form can be written as 
\begin{equation}
A(u,v)=\sum_{T\in \tau_{h}}A^{T}(u,v) \nonumber \\
\end{equation}
where $A^{T}(\cdot,\cdot)$ is restriction of $A(\cdot,\cdot)$ on each triangle 
$T \in \tau_{h}$. 
\begin{equation}
A^{T}(u,v)=\int_{T} \epsilon \, \nabla u \cdot \nabla v+b^{T}(u,v)+\int_{T}c\ u \, v+b^{T,stab}(u,v) \nonumber
\end{equation} 
where
\begin{equation}
b^{T}(u,v)=\frac{1}{2} \left[\int_{T}(\vec{b} \cdot \nabla u)v-\int_{T}(\vec{b} \cdot \nabla v)u-\int_{T}(\nabla \cdot \vec{b}) \, u \, v \right] 
\label{new2}
\end{equation}

$b^{T}(u,v)$ can be split into two parts $b^{T,sym}(u,v)$ and $b^{T,skew}(u,v)$
where
\begin{equation}
b^{T,sym}(u,v)=\frac{1}{2} \int_{T} (\nabla \cdot \vec{b}) uv  \nonumber 
\end{equation}

\begin{equation}
b^{T,skew}(u,v)=\frac{1}{2} \left[\int_{T}(\vec{b} \cdot \nabla u)v-\int_{T}(\vec{b} \cdot \nabla v)u \right]
\end{equation}

and
\begin{equation}
b^{T,stab}(u,v)=\int_{T}(-\epsilon \, \Delta u+\vec{b} \cdot \nabla u+c \, u)\ \delta_{T}\ \vec{b} \cdot \nabla v \nonumber \\
\end{equation}

and right hand side function $F(\cdot)$ can be written as
\begin{equation}
F(v)=\int_{\Omega} fv+\sum_{T \in \tau_{h}} \int_{T}f \ \delta_{T} \, \vec{b} \cdot \nabla v \nonumber
\end{equation}

We introduce the nonconforming virtual finite element method for the model problem
(\ref{temp2}) which read as, find $u_{h}\in V^{k}_{h}$ such that 
\begin{equation}
A_{h}(u_{h},v_{h})=F_{h}(v_{h}) \quad \forall \  v_{h} \in V^{k}_{h} 
\label{temp5}
\end{equation}
where $V^{k}_{h} \subset H^{1,nc}(\tau_{h},k)$ is a global nonconforming virtual finite 
element space.

We will refer local virtual element space by $V^{k}_{h}(T)$. For 
$k\geq 1$ the finite dimensional  space $V^{k}_{h}(T)$ associated to the polygon $T$ is
given by
\begin{equation}
V^{k}_{h}(T):= \left \{ v\in H^{1}(T):\frac{\partial v}{\partial\mathbf{n}} \in \mathbb{P}^{k-1}(e)\ \forall \ e \subset \partial T,\ \Delta v \in \mathbb{P}^{k-2}(T) \right \}
\label{temp6}
\end{equation}

For each polygon $T$, the dimension of $V^{k}_{h}(T)$ is given by
\begin{center}
$nk+(k-1)k/2 $ \ where \  $T \subset \Omega \subset \mathbb{R}^{2}$
\end{center}

\begin{center}
$nk(k+1)/2+(k-1)k(k+1)/6$ \ where \  $T \subset \Omega \subset \mathbb{R}^{3}$
\end{center}

Now we shall introduce degrees of freedom

For $l\geq 0$, $M^{l}(e)$ and $ M^{l}(T)$ respectively denote the set of scaled monomials on $e$ and $T$
\begin{equation}
M^{l}(e)=\left\{ \left(\frac{x-x_{e}}{h_{e}} \right)^{s}, \vert s \vert \leq l \right \} \nonumber
\end{equation}
and 
\begin{equation}
M^{l}(T)= \left\{ \left(\frac{x-x_{T}}{h_{T}} \right)^{s}, \vert s \vert \leq l \right\} \nonumber
\end{equation}
 Using scaled monomial as basis function we define degrees of freedom
 
 (L1) all the moments of $v_{h}$ of order upto $k-1$ on each edge/face $e \subset \partial T$
\begin{equation} 
 \mu^{k-1}_{e}(v_{h})=\left\{ \frac{1}{|e|} \int_{e} v_{h} \, m \, ds, \ \forall m \in M^{k-1}(e) \right\} 
 \label{temp7}
\end{equation}
 
 (L2) all the moments of $v_{h}$ of order upto $(k-2)$ on $T$	
\begin{equation}
\mu^{k-2}_{T}(v_{h})=\left \{ \frac{1}{|T|} \int_{T}v_{h} \, m \, dT, \quad \forall \ m\in M^{k-2}(T) \right \} 
\label{temp8}
\end{equation}

The degrees of freedom(\ref{temp7})-(\ref{temp8}) are unisolvent for $V^{k}_{h}(T)$, which 
is proved in detail in \cite{de2014nonconforming}.

\subsection{Global nonconforming virtual element space $V^{k}_{h}$ \label{sec4}}
 We have defined local nonconforming virtual element space for each element $T \in \tau_{h}$.The global nonconforming virtual element space $V^{k}_{h}$ of order $k$ is given by 
 
\begin{center}
$n_{edg} \, k+n_{ele} \, (k-1) \, k/2$ \ for \ $d=2$
\end{center} 

\begin{center}
$n_{faces} \, k \, (k+1)/2+n_{ele} \, (k-1) \, k \, (k+1)/6$ \ for \ $d=3$ 
\end{center}
where $n_{edg}$ and $n_{faces}$ denote total no of edges ($d=2$) and faces($d=3$) respectively.\ $n_{ele}$ denotes total no of elements in $\tau_{h}$
 \begin{equation}
 V^{k}_{h}=\{ v_{h} \in H^{1,nc}(\tau_{h};k):v_{h}|_{T} \in V^{k}_{h}(T) \  \forall \ T \in \tau_{h} \} \nonumber
\end{equation}
 
\subsection{Interpolation operator \label{sec5}}
Assuming assumption (Z1)-(Z3), there exists a local polynomial approximation 
$u_{\Pi} \in \mathbf{P}^{k}(T)$, which satisfies the following approximation property 
\begin{equation}
\| u-u_{\Pi} \|_{0,T}+h_{T} \, |u-u_{\Pi}|_{1,T} \leq C \, h^{s}_{T} \, |u|_{s,T} 
\label{aprox1} 
\end{equation}  
where   $u \in H^{s}(T)$, $2\leq s\leq k+1$, and $C$ is a positive constant independent of $h_{T}$ and  depends on regularity constant $\rho$.

In standard finite element method literature we have seen interpolation operator 
depending on polynomial basis function rigorously. But in the case of virtual element theory 
we do not have explicit knowledge of the basis function. In this paragraph we define an 
interpolation operator on $V^{k}_{h}$ having optimal approximation properties using 
degrees of freedom only without explicit knowledge of basis function. Since 
detail construction of interpolation operator has been shown in \cite{de2014nonconforming,ciarlet1991basic}, we state 
here only the result.

For every $v\in H^{s}(T)$, there exists an unique interpolant $v_{I}\in V^{k}_{h}$ 
satisfies 
\begin{equation}
\|v-v_{I} \|_{0,T}+h_{T} \, \|v-v_{I}\|_{1,T} \leq C \, h_{T}^{s} \, \|v\|_{s,T} 
\label{aprox2}
\end{equation}
where $2\leq s\leq k+1$ and $T\in \tau_{h}$ be an arbitrary element.     
 
 \subsection{Construction of $A_{h} $ \label{sec6}}
In this section we explicitly describe the discrete bilinear form $A_{h}(u_{h},v_{h})$ 
and right hand side function $F_{h}$. As we mentioned earlier if the model 
problem (\ref{temp2}) is convection dominated then we need to add additional diffusion 
in streamline direction. In the virtual element formulation we rewrite the discrete bilinear 
form into two parts one is polynomial part constructed by various type of projection 
operators like elliptic projection operator $\Pi^{\nabla}_{k}:V^{k}_{h} \rightarrow \mathbb{P}^{k}(T)$ 
or $L^{2}$ projection $\Pi_{k}:V^{k}_{h}\rightarrow \mathbb{P}^{k}(T)$ and another is 
stabilization part which is responsible to stabilize the bilinear form. We have 
added additional term $\int_{T}(\vec{b} \cdot \nabla u)(\vec{b} \cdot \nabla v) \, dT$ 
to the bilinear form $A^{T}_{h}(u,v)$ to reduce the oscillation. In virtual element 
formulation we will reveal the stabilization part into two parts one is polynomial part 
and another is stabilization part, same as what we do for diffusion or reaction part. The 
present framework is stable for very small value of $\epsilon$. Convergence analysis
independent of $\epsilon$ will be shown in section \ref{sec14}.
 
Let us write the discrete bilinear form $A_{h}(u_{h},v_{h})$ as
\begin{equation}
 A_{h}(u_{h},v_{h})=\sum_{T \in \tau_{h}} A^{T}_{h}(u_{h},v_{h}) \ \forall \ u_{h},v_{h} \in V^{k}_{h}
\end{equation}   

$A^{T}_{h}:V^{k}_{h}\times V^{k}_{h}\rightarrow \mathbb{R}$ denoting the restriction of $A_{h}(u_{h},v_{h})$ to the local space $V^{k}_{h}(T)$.

The bilinear form $A^{T}_{h}$ can be decomposed as
\begin{eqnarray}
A^{T}_{h}(u_{h},v_{h})&:=&a^{T}_{h}(u_{h},v_{h})+b^{T}_{h}(u_{h},v_{h}) \nonumber \\
 &+&c^{T}_{h}(u_{h},v_{h})+b^{T,stab}_{h}(u_{h},v_{h})
\label{new1}
\end{eqnarray}

where
\begin{equation}
a^{T}_{h}(u_{h},v_{h}):=\int_{T} \epsilon \, \Pi_{k-1}(\nabla u_{h}) \cdot \Pi_{k-1}(\nabla v_{h}) +s^{T}_{a}((I-\Pi_{k})u_{h},(I-\Pi_{k})v_{h})
\end{equation}

\begin{equation}
b^{T}_{h} (u_{h},v_{h}):= -b^{T,sym}_{h}(u_{h},v_{h})+b^{T,skew}_{h}(u_{h},v_{h}) \nonumber
\end{equation}
symmetric part is defined as

\begin{equation}
b^{T,sym}_{h} (u_{h},v_{h}):=\frac{1}{2} \int_{T}(\nabla \cdot \vec{b}) \, \Pi_{k}(u_{h}) \, \Pi_{k}(v_{h}) \, dT+s^{T,sym}((I-\Pi_{k})u_{h},(I-\Pi_{k})v_{h}) \, dT \nonumber
\end{equation}

skew-symmetric part of convection term is defined as
\begin{equation}
b^{T,skew}_{h}(u_{h},v_{h}):=\frac{1}{2} \left[\int_{T}\vec{b} \cdot \Pi_{k-1}(\nabla u_{h}) \, \Pi_{k}(v_{h}) \, dT -\int_{T} \vec{b} \cdot \Pi_{k-1}(\nabla v_{h})\Pi_{k}(u_{h}) \, dT \right] 
\end{equation}

\begin{equation}
c^{T}_{h}(u_{h},v_{h}) := \int_{T}c \, \Pi_{k}(u_{h}) \, \Pi_{k}(v_{h}) \, dT 
                        + s^{T}_{c}((I-\Pi_{k})u_{h},(I-\Pi_{k})v_{h}) \nonumber
\end{equation}

stabilization term
\begin{eqnarray}
b^{T,stab}_{h}(u_{h},v_{h})&:=& \int_{T}(-\epsilon \, \Pi_{k-2}(\Delta u)+c \, \Pi_{k}(u))\ \delta_{T}(\vec{b} \cdot \Pi_{k-1}(\nabla v)) \ dT \nonumber \\
                           & + & \int_{T}(\vec{b} \cdot \Pi_{k-1}(\nabla u))(\vec{b} \cdot \Pi_{k-1}(\nabla v))\ dT \nonumber \\
                           &+&s^{T,stab}_{b}((I-\Pi_{k})u_{h},(I-\Pi_{k})v_{h}) 
\end{eqnarray}
$ s^{T}_{a}, s^{T}_{c},s^{T,stab}_{b},s^{T,sym} $ are the stabilization terms. These 
terms are symmetric and vanish on the polynomial space $\mathbb{P}_{k}(T)$.

\section{Discrete stability \label{sec7}}
Our model problem (\ref{temp2}) contains three parts diffusion, reaction, convection. Weak 
formulation of diffusion and reaction parts are symmetric hence stability analysis of these 
two parts follows same as discussed in the literature of 
VEM \cite{beirao2013basic,brezzi2013virtual,de2014nonconforming}, convection part is not 
symmetric which does not follow same stability analysis as diffusion and reaction parts. We 
need some extra effort to stabilize convection part. In discrete formulation 
$A^{T}_{h}(u_{h},v_{h})$,  we add additional term 
$\int_{T} (\vec{b} \cdot \nabla u_{h}) \, (\vec{b} \cdot \nabla v_{h})$ to capture non physical 
oscillation produced by convection term $\vec{b} \cdot \nabla u$ in convection dominated 
region $(\epsilon <<1)$. Fortunately 
$\int_{T} (\vec{b} \cdot \nabla u_{h}) \, (\vec{b} \cdot \nabla v_{h})$ is symmetric and we can 
stabilize the term using same technique as diffusion and reaction parts, which prevent 
drastic change of $\vec{b}.\nabla u_{h}$ in convection dominated region. Finally 
we can conclude that there exist positive 
constants $\alpha_{\ast}, \alpha^{\ast}, \gamma_{\ast},\gamma^{\ast}, s_{\ast},s^{\ast}, \Gamma_{\ast}, \Gamma^{\ast} $ such that 
\begin{equation}
\alpha_{\ast} \, a^{T}(v_{h},v_{h}) \leq a^{T}_{h}(v_{h},v_{h}) \leq \alpha^{\ast} \, a^{T} (v_{h},v_{h}) 
\label{stbl1} 
\end{equation}  

\begin{equation}
\gamma_{\ast} \, c^{T}(v_{h},v_{h}) \leq c^{T}_{h}(v_{h},v_{h}) \leq \gamma^{\ast} \, c^{T}(v_{h},v_{h})
\label{stbl2}
\end{equation}

\begin{equation}
\Gamma_{\ast} \, G^{T}(v_{h},v_{h}) \leq G^{T}_{h}(v_{h},v_{h}) \leq \Gamma^{\ast} \, G^{T}(v_{h},v_{h})
\label{stbl3}
\end{equation} 

\begin{equation}
s_{\ast} \, b^{T,sym}(v_{h},v_{h}) \leq b^{T,sym}_{h}(v_{h},v_{h}) \leq s^{\ast} b^{T,sym}(v_{h},v_{h})
\label{stbl4} 
\end{equation}
for all $v_{h}\in V^{k}_{h}(T)$ where $T$ is an arbitrary element and $G^{T}(\cdot,\cdot), G^{T}_{h}(\cdot,\cdot)$ 
are defined by
\begin{eqnarray}
 G^{T}(u_{h},v_{h})&:=&\int_{T} \delta_{T} \, (\vec{b} \cdot \nabla u_{h})\, (\vec{b} \cdot \nabla v_{h}) \nonumber \\
G^{T}_{h}(u_{h},v_{h})&:=&\int_{T} \delta_{T} \, \vec{b} \cdot \Pi_{k-1}(\nabla u_{h}) \ \vec{b} \cdot \Pi_{k-1}(\nabla v_{h}) \nonumber \\
&+& s^{T,stab}_{b}((I-\Pi_{k})u_{h}, (I-\Pi_{k})v_{h})
\end{eqnarray}

\subsection{Polynomial consistency \label{sec8}}
\begin{lem}
Let $u_{h}|_{T} \in \mathbb{P}^{k}(T)$ and $v_{h}|_{T}\in H^{2}(T)$. Then the bilinear 
form $A_{h}^{T}(u_{h},v_{h})$ defined in (\ref{new1}) satisfies polynomial consistency 
property, i.e. $ A_{h}^{T}(u_{h},v_{h})=A^{T}(u_{h},v_{h})$ for all $h>0$ and for all 
$T \in \tau_{h}$.
\label{pc1}
\end{lem}
 
\begin{proof}
 
If $u_{h}$ or $v_{h}$, or both are polynomial of degrees $k$, then the stabilization 
terms $s^{T}_{a},s^{T}_{c},s^{T,stab}_{b}$ vanish.Now we will prove the following   
 
\begin{eqnarray}
a^{T}_{h}(u_{h},v_{h}) &=& a^{T}(u_{h},v_{h}) \nonumber\\
b^{T,skew}_{h}(u_{h},v_{h}) &=& b^{T,skew}(u_{h},v_{h}) \nonumber \\
b^{T,sym}_{h}(u_{h},v_{h}) &=& b^{T,sym}(u_{h},v_{h}) \nonumber \\
c^{T}_{h}(u_{h},v_{h}) &=& c^{T}(u_{h},v_{h}) \nonumber \\
b^{T,stab}_{h}(u_{h},v_{h}) &=& b^{T,stab}(u_{h},v_{h}) \nonumber
\end{eqnarray}

$L^{2}$ projection $ \Pi_{k}$ is invariant on polynomial space $\mathbb{P}^{k}(T)$, i.e., $\Pi_{k}(p)=p$, where $p\in \mathbb{P}^{k}$. Using orthogonality property of $L^{2}$ projection we can estimate required polynomial consistency property   

\begin{eqnarray}
a^{T}_{h}(p,v_{h})&=& \int_{T} \epsilon \, \nabla p \cdot \Pi_{k-1}(\nabla v_{h}) \, dT \nonumber \\
                  &=& \int_{T} (\Pi_{k-1}(\nabla v_{h})-\nabla v_{h}) \, \epsilon \, \nabla p \, dT \nonumber \\
                  &+& \int_{T} \epsilon \, \nabla p \cdot \nabla v_{h} \, dT \nonumber \\
                  &=& \int_{T} \epsilon \, \nabla p \, \nabla v_{h} \, dT \nonumber \\
                  &=& a^{T}(p,v_{h})  
\end{eqnarray}

$b^{T}_{h}(p,v_{h}) $ contains three parts. Since $ \nabla p$ is a polynomial of 
degrees $k-1$  less than deg$(\Pi_{k-1}(v_{h}))$, using orthogonality property of $L^{2}$ 
operator we estimate polynomial consistency of first term
\begin{eqnarray}
b^{T,skew}_{h}(p,v_{h}) &=&\frac{1}{2} \left[\int_{T} (\vec{b} \cdot \nabla p) \, \Pi_{k}(v_{h}) \, dT-\int_{T} \vec{b} \cdot \Pi_{k-1}(\nabla v_{h}) \,p \, dT \right]\nonumber \\
\label{pc2}
\end{eqnarray}

\begin{eqnarray}
\int_{T} \vec{b} \cdot \nabla p\ \Pi_{k}(v_{h}) \, dT &=& \int_{T} \vec{b} \cdot \nabla p \, (\Pi_{k}(v_{h})-v_{h}) \, dT \nonumber\\ 
                                             &+&\int_{T} \vec{b} \cdot \nabla p \, v_{h} dT\nonumber\\
                                             &=&\int_{T}\vec{b} \cdot \nabla p \, v_{h} \ dT
\end{eqnarray}
$\Pi_{k-1}(\nabla v_{h})$ is a polynomial of degree $k-1$, hence we cannot apply 
orthogonality property of $L^{2}$ projection. We use polynomial approximation 
property of $\nabla v_{h}$  and $H^{2}(T)$ regularity of $v_{h}$ to establish the 
following result,

\begin{eqnarray}
\int_{T} \vec{b} \, \Pi_{k-1}(\nabla v_{h}) \, p &=& \int (\Pi_{k-1}(\nabla v_{h})-\nabla v_{h}) \, \vec{b} \, p \nonumber \\
                                                      & + &\int_{T} \vec{b} \cdot \nabla v_{h} \, p \nonumber \\
                                                      & \leq & \|\vec{b}\|_{\infty,T} \, \| p \|_{0,T} \, \|\nabla v_{h}-\Pi_{k-1}(\nabla v_{h}) \| +\int_{T} \vec{b} \, \nabla v_{h} \, p \nonumber \\
                                                      & \leq & C \| \vec{b} \|_{\infty,T} \, h_{T} \, \vert \nabla v_{h} \vert_{1,T} \| p \|_{0,T} \nonumber \\
                                                      & +& \int_{T} \vec{b} \cdot \nabla v_{h} \, p \nonumber \\
                                                      & \approx & \int_{T} \vec{b} \cdot \nabla v_{h} \, p                                                     
                                                      \label{tempc3}
\end{eqnarray}
for very small values of $h_{T}$.

Again using orthogonality property of $ L^{2}$ projection we establish 
\begin{eqnarray}
b^{T,sym}_{h}(p,v_{h})&=& \frac{1}{2}\int_{T} (\nabla \cdot \vec{b}) \, p \, \Pi_{k}(v_{h})\ dT \nonumber \\ 
&=& \frac{1}{2} \int_{T} (\nabla \cdot \vec{b})(\Pi_{k}(v_{h})-v_{h})p + \frac{1}{2} \int_{T} (\nabla \cdot \vec{b}) \, p \, v_{h} \nonumber\\
                                                          &=& \frac{1}{2}\int_{T} (\nabla \cdot \vec{b}) \, p \, v_{h} 
                                                          \label{tempc4}
\end{eqnarray}

and 
\begin{equation}
c^{T}_{h}(p,v_{h})=c^{T}(p,v_{h})
\label{tempc5}
\end{equation}

Now we will discuss polynomial consistency of additional terms. $\Delta p$ is a polynomial 
of degree $k-2$ less than deg $(\Pi_{k-1}(\nabla v)) $, hence we use orthogonality property of $L^{2}$ function to establish the following results.  

\begin{eqnarray}
b^{T,stab}_{h}(u_{h},v_{h}) &=& \int_{T}(-\epsilon \, \Delta p+c \, p) \, \delta_{T} \, (\vec{b} \cdot \Pi_{k-1}(\nabla v_{h})) \, dT \nonumber\\
                            &+& \int_{T} \delta_{T} \, (\vec{b} \cdot \nabla p) \, (\vec{b} \cdot \Pi_{k-1}(\nabla v_{h})) \, dT 
\end{eqnarray}

\begin{eqnarray}
\int_{T} -\epsilon \, \Delta p \, \delta_{T} (\vec{b} \cdot \Pi_{k-1}(\nabla v_{h})) \, dT & = &  \delta_{T} \int_{T} -\epsilon \, \Delta p \, (\vec{b} \cdot \Pi_{k-1}(\nabla v_{h})) \, dT \nonumber \\
&=& \delta_{T} \int_{T}-\epsilon \, \Delta p \, (\vec{b} \cdot \Pi_{k-1}(\nabla v_{h})-\vec{b} \cdot \nabla v_{h}) \, dT \nonumber \\
&+& \delta_{T} \int_{T}-\epsilon  \, \Delta p \, (\vec{b} \cdot \nabla v_{h}) \, dT \nonumber \\
&=& \delta_{T} \int_{T}-\epsilon  \, \Delta p \, (\vec{b} \cdot \nabla v_{h}) \, dT
\end{eqnarray}

Using same technique as (\ref{tempc3}) we establish the following result

\begin{eqnarray}
\int_{T} c \, p \, \delta_{T} \, (\vec{b} \cdot \Pi_{k-1}(\nabla v_{h})) &=& \delta_{T} \int_{T} c \, p \, (\vec{b} \cdot \Pi_{k-1}(\nabla v_{h})-\vec{b} \cdot \nabla v_{h}) \nonumber \\
&+&\delta_{T}\int_{T} c \, p \, \vec{b} \cdot \nabla v_{h}  \nonumber \\
 &\leq & \delta_{T} \, c_{\text{max}} \, C \, \|p\|_{0,T} \, \|\vec{b}\|_{0,\infty} \, h_{T} \, |\nabla v_{h} |_{1,T} \nonumber \\
 &+& \delta_{T} \int_{T} c \, p \, \vec{b} \cdot \nabla v_{h} \nonumber \\
 & \approx & \delta_{T} \int_{T} c \, p \, \vec{b} \cdot \nabla v_{h}
 \label{new_aprox}
\end{eqnarray}
for very small values of $h_{T}$.

Again using  orthogonality property of $L^{2}$ function we derive the following result. 

\begin{eqnarray}
\int_{T}(\vec{b} \cdot \nabla p) \, \delta_{T} \, (\vec{b} \cdot \Pi_{k-1}(\nabla v_{h})) \, dT &=& \delta_{T} \, \int_{T}(\vec{b} \cdot \nabla p)(\vec{b} \cdot \Pi_{k-1}(\nabla v_{h})-\vec{b} \cdot \nabla v_{h}) \, dT \nonumber \\ 
& + & \delta_{T} \int_{T} (\vec{b} \cdot \nabla p)(\vec{b} \cdot \nabla v_{h} ) \, dT \nonumber \\
&=& \delta_{T} \int_{T}(\vec{b} \cdot \nabla p)(\vec{b} \cdot \nabla v_{h} ) \, dT
\end{eqnarray}

\end{proof}
\subsection{Remark:}
In estimation (\ref{tempc3}) we have approximated the 
term $\int_{T}\vec{b} \, \Pi_{k-1}(\nabla v_{h}) \, p $ by $\int_{T} \vec{b} \cdot \nabla v_{h} \, p$. Hence 
corresponding error is 
$\bigg|\int_{T}\vec{b} \, \Pi_{k-1}(\nabla v_{h}) \, p -\int_{T} \vec{b} \cdot \nabla v_{h} \, p \bigg| \leq C \, \|\vec{b}\|_{\infty,T} \,  h_{T} \,  |\nabla v_{h} |_{1,T} \  \|p \|_{0,T} $. Since we 
have assumed that $v_{h}|_{T} \in H^{2}(T) $, $  |\nabla v_{h}|_{1,T} $ and other terms are 
well defined. Therefore making mesh diameter $h_{T}$ sufficiently small we can reduce the 
error less than any positive quantity $\epsilon_{1}$ which is  different from the
diffusion coefficient $\epsilon$. Hence for sufficiently small size of diameter of mesh 
element $T$ the error is negligible. In estimation (\ref{new_aprox}) we have followed 
same methodology.

\section{Discrete coercivity \label{sec9}}
Now we discuss the coerciveness of $A_{h}(u_{h},v_{h})$ on $V^{k}_{h}$. We assume the 
assumptions (Z1)-(Z3) and (\ref{asmp1}) hold then there exist 
constants $\mu_{1}$ and $\mu_{2}$ independent of $T \in \tau_{h}$ such that the following 
local inverse inequality hold

\begin{equation}
\| \Delta v_{h} \|_{0,T} \leq \mu_{1} \, h^{-1}_{T} \, |v_{h}|_{1,T} \quad \forall \ v_{h}\in V_{h}^{k}, T \in \tau_{h}
\label{dis_coer5}
\end{equation}

\begin{equation}
|v_{h}|_{1,T} \leq \mu_{2} \, h^{-1}_{T} \, \|v_{h}\|_{0,T} \quad \forall v_{h}\in V^{k}_{h},  T \in \tau_{h}
\label{dis_coer6}
\end{equation}
Let us introduce domain dependent norm 

\begin{equation}
|||v_{h}|||:=\sum_{T} \left\{ \epsilon \, |v_{h}|^{2}_{1,T}+c_{0} \, \|v_{h}\|^{2}_{0,T}+\delta_{T} \, \|\vec{b} \cdot \nabla v_{h} \|^{2}_{0,T} \right\} 
\label{dis_coer1}
\end{equation}

We assume the  following assumption on control parameter $\delta_{T}$ 
\begin{equation}
0<\delta_{T} \leq \text{min} \left\{ \frac{c_{0}\ \text{min}\{s^{\ast},\gamma_{\ast}\}}{4 \, c^{2}_{\text{max}}}, \frac{h^{2}_{T} \, \alpha_{\ast}}{2 \, \epsilon \, \mu^{2}_{1}},\text{min} \left\{1,\frac{1}{c_{I}} \right\} \frac{c_{0} \, \text{min}\left\{s^{\ast},\gamma_{\ast}\right\} h_{T}^{2}}{4 \, \| \vec{b}\|^{2}_{0,\infty} \, \mu^{2}_{2}} \right \} 
\label{dis_coer7}
\end{equation}

\begin{lem}
\label{dis_coer2}
Let the virtual element space $V^{k}_{h}$ satisfies the assumptions $(Z1), (Z2), (Z3) \& (\ref{asmp1})$. Let 
the virtual element space satisfies the condition (\ref{aprox2}). Then the discrete 
bilinear form $A_{h}(u_{h},v_{h})$ is coercive on $V^{k}_{h}$, i.e. 
\begin{center}
$A_{h}(v_{h},v_{h}) \geq \alpha |||v_{h}|||^{2}$
\end{center} 
where $\alpha$ is a positive constant.
\end{lem}

\begin{proof}
Let us consider the bilinear form $A_{h}(\cdot,\cdot)$.

\begin{equation}
A_{h}(u_{h},v_{h})=\sum_{T} A^{T}_{h}(u_{h},v_{h}) \nonumber
\end{equation}


\begin{eqnarray}
A^{T}_{h}(v_{h},v_{h})& =& a^{T}_{h}(v_{h},v_{h}) +b^{T}_{h}(v_{h},v_{h})+c^{T}_{h} \nonumber\\
                      & +& \int_{T}[-\epsilon \, \Pi_{k-2}(\Delta v_{h})+c \, \Pi_{k}(v_{h})] \, \delta_{T} \, (\vec{b} \cdot \Pi_{k-1}(\nabla v_{h})) \nonumber \\
                      & + &\int_{T}(\vec{b} \cdot \Pi_{k-1}(\nabla v_{h})) \, \delta_{T} \, (\vec{b} \cdot \Pi_{k-1}(\nabla v_{h})) \nonumber\\
                      & + & s^{T,stab}_{h}((I-\Pi_{k})v_{h},(I-\Pi_{k})v_{h}) 
                      \label{dis_coer3}
\end{eqnarray}
Using stability property of $ a^{T}_{h}(v_{h},v_{h}), G^{T}_{h}(v_{h},v_{h})$, $c^{T}_{h}(v_{h},v_{h}) $ and $ b^{T,sym}_{h}$ and considering the
assumptions mentioned in the lemma we can write 
\begin{eqnarray}
A^{T}_{h}(v_{h},v_{h}) &\geq & \alpha_{\ast} \, a^{T}(v_{h},v_{h})+\text{min} \{s^{\ast},\gamma_{\ast} \} \, c_{0} \, \|v_{h} \|^{2}_{0,T} + \Gamma_{\ast} \, \delta_{T} \, \|\vec{b} \cdot \nabla v_{h}\|^{2}_{0,T}  \nonumber\\
                       &+&\int_{T} [-\epsilon \, \Pi_{k-2}(\Delta v_{h})+c \, \Pi_{k}(v_{h})] \, \delta_{T} (\vec{b} \cdot \Pi_{k-1}(\nabla v_{h})) 
                       \label{dis_coer4}
\end{eqnarray}

Using boundedness property of projection operator and satisfying the 
condition (\ref{dis_coer7}) we estimate the additional term.

\begin{eqnarray}
\bigg| \int_{T} -\epsilon \, \Pi_{k-2}(\Delta v_{h}) \, \delta_{T} \, (\vec{b} \cdot \Pi_{k-1}(\nabla v_{h})) \bigg| & \leq & \delta_{T} \, \epsilon \, \|\Delta v_{h} \|_{0,T} \ \|\vec{b} \cdot \nabla v_{h} \|_{0,T} \nonumber\\
& \leq & \delta_{T} \big[\frac{\epsilon^{2}}{2} \|\Delta v_{h}\|^{2}_{0,T} \nonumber \\
&+ & \frac{1}{2} \|\vec{b}.\nabla v_{h} \|^{2}_{0,T} \big]
\end{eqnarray}
using inverse inequality property of VE space (\ref{dis_coer5}) we get,

\begin{eqnarray}
\frac{\delta_{T}}{2} \, \epsilon^{2} \, \| \Delta v_{h} \|^{2}_{0,T} &\leq & \frac{\delta_{T}}{2} \, \epsilon^{2} \, \mu^{2}_{1} \, h^{-2}_{T} |v_{h} |^{2}_{1,T} \nonumber \\
&\leq & \frac{\epsilon}{4} \, \alpha_{\ast} \, |v_{h}|^{2}_{1,T}
\end{eqnarray}
using polynomial approximation property (\ref{aprox1}) of virtual element space and 
assuming control parameter $\delta_{T}$ satisfying the condition (\ref{dis_coer7}) we 
can estimate   

\begin{eqnarray}
\int_{T} c \, \Pi_{k}(v_{h}) \, \delta_{T} \, (\vec{b} \cdot \Pi_{k-1}(\nabla v_{h})) &=& \int_{T} c \, v_{h} \, \delta_{T} \, (\vec{b} \cdot \Pi_{k-1}(\nabla v_{h})) \nonumber \\
&=& \int_{T} c \, v_{h} \, \delta_{T} \, \vec{b} \cdot [\Pi_{k-1}(\nabla v_{h})-(\nabla v_{h})] \nonumber \\
&+&\int_{T}c \, v_{h} \, \delta_{T} \, (\vec{b} \cdot \nabla v_{h})
\label{dis_coer8}
\end{eqnarray}
Since $c$ is smooth enough we can bound it by $c_{\text{max}}$ 
\begin{eqnarray}
\bigg| \int_{T}c \, v_{h} \, \delta_{T} \, (\vec{b} \cdot \nabla v_{h}) \bigg| & \leq &  \delta_{T} \, \| c \, v_{h}\|_{0,T} \, \|\vec{b} \cdot \nabla v_{h} \|_{0,T} \nonumber \\
& \leq & \delta_{T} \left[\frac{1}{2} \|c \, v_{h}\|^{2}_{0,T}+ \frac{1}{2} \|\vec{b} \cdot \nabla v_{h} \|^{2}_{0,T} \right] \nonumber \\
&=& \frac{\delta_{T}}{2} \|c \, v_{h} \|^{2}_{0,T} +\frac{\delta_{T}}{2} \|\vec{b} \cdot \nabla v_{h} \|^{2}_{0,T} 
\label{dis_coer9}
\end{eqnarray}

If the control parameter $\delta_{T}$ satisfies the condition (\ref{dis_coer7}) we can estimate 
\begin{eqnarray}
\frac{\delta_{T}}{2} \|c \, v_{h} \|^{2}_{0,T} &\leq & \frac{\delta_{T}}{2} c_{\text{max}}^{2} \| v_{h} \|^{2} \nonumber \\
 &\leq & \frac{c_{0}}{8} \text{min}\{s^{\ast},\gamma_{\ast} \} \|v_{h} \|^{2}_{0,T}
 \label{dis_coer10}
\end{eqnarray}
We consider vector valued convection coefficient $\vec{b}$ is regular enough. Taking it 
outside of integration and bounding it by $L_{\infty}$ norm $|\vec{b}|_{0,\infty}$ and 
using (\ref{dis_coer6}), we can write
\begin{eqnarray}
\frac{1}{2} \, \delta_{T} \, \| \vec{b} \cdot \nabla v_{h} \|^{2}_{0,T} &\leq &  \frac{1}{2} \, \delta_{T} \, |\vec{b} |^{2}_{0,\infty} \| \nabla v_{h} \|^{2}_{0,T} \nonumber \\
& \leq & \frac{1}{2} \, \delta_{T} \, |\vec{b} |^{2}_{0,\infty} \, \mu^{2}_{2} \, h_{T}^{-2} \, \|v_{h} \|^{2}_{0,T} \nonumber \\
& \leq & \frac{c_{0}}{8} \, \text{min} \{ s^{\ast},\gamma_{\ast} \} \| v_{h}\| ^{2}_{0,T}
\label{dis_coer11}
\end{eqnarray}
 First term of (\ref{dis_coer8}) can be estimated as
\begin{eqnarray}
\bigg| \int_{T} c v_{h} \, \delta_{T} \, \vec{b} \cdot (\Pi_{k-1}(\nabla v_{h})-(\nabla v_{h})) \bigg|  & \leq &  \delta_{T} \, \|c \, v_{h} \|_{0,T} \| \vec{b} \cdot (\Pi_{k-1}(\nabla v_{h})-\nabla v_{h}) \|_{0,T}  \nonumber\\
&\leq & \delta_{T} \big[\frac{1}{2} \|c \, v_{h} \|^{2}_{0,T} \nonumber \\
& + & \frac{1}{2} \| \vec{b} \cdot (\Pi_{k-1}(\nabla v_{h})-\nabla v_{h}) \|_{0,T}^{2} \big]
\label{dis_coer12}
\end{eqnarray}

\begin{eqnarray} 
\delta_{T} \, \frac{1}{2} \, \|c \, v_{h} \|^{2}_{0,T} &\leq & \frac{c_{0}}{8} \, \text{min}\{s^{\ast},\gamma_{\ast}\} \|v_{h}\|^{2}_{0.T}
\label{dis_coer13}
\end{eqnarray}
 Using polynomial approximation property (\ref{aprox1}) and local inverse inequality (\ref{dis_coer6}) of VE space we can write
\begin{eqnarray} 
\delta_{T} \, \frac{1}{2} \| \vec{b} \cdot (\Pi_{k-1}(\nabla v_{h})-\nabla v_{h}) \|_{0,T}^{2} &\leq & \frac{1}{2} \, \delta_{T} \, \|\vec{b}\|^{2}_{0,\infty} \, c_{I} \, \| \nabla v_{h} \|^{2}_{0,T}  \nonumber \\
&\leq & \frac{1}{2} \, \delta_{T} \, \| \vec{b}\|^{2}_{0,\infty}  \, c_{I} \, \mu^{2}_{2} \, h^{-2}_{T} \|v_{h}\|^{2}_{0,T} \nonumber \\
& \leq &  \frac{1}{8} \, c_{0} \, \text{min} \{ s^{\ast},\gamma_{\ast} \} \|v_{h}\|^{2}_{0,T}
\label{dis_coer14} 
\end{eqnarray}

Finally we estimate local bilinear form $ A^{T}_{h}(v_{h},v_{h})$
\begin{eqnarray}
A^{T}_{h}(v_{h},v_{h}) & \geq & \frac{3}{4} \, \alpha_{\ast}  \, \epsilon \, |v_{h}|^{2}_{1,T} \nonumber \\
& + &  \frac{3}{8} \, \text{min} \{s^{\ast},\gamma_{\ast} \}  \, c_{0} \, \|v_{h} \|^{2}_{0,T} \nonumber \\
& + & \Gamma_{\ast} \, \delta_{T} \, \|\vec{b} \cdot \nabla v_{h} \|^{2}_{0,T} 
\label{dis_coer15}
\end{eqnarray}
 
Using local estimation(\ref{dis_coer15}) we can write
\begin{eqnarray}
A_{h}(v_{h},v_{h}) & \geq & \sum_{T} A^{T}_{h}(v_{h},v_{h}) \nonumber \\
                   & \geq & \alpha \sum_{T} \left\{ \epsilon \, |v_{h}|^{2}_{1,T} +c_{0} \, \|v_{h} \|^{2}_{0,T}+\delta_{T} \, \|\vec{b} \cdot \nabla v_{h} \|^{2}_{0,T} \right \} \nonumber \\
                   & \geq & \alpha \, |||v_{h}|||^{2}
                   \label{dis_coer16}
\end{eqnarray}
Positive constant $\alpha$ is defined by $\alpha$:= min $\{ \alpha_{1},\alpha_{2} ,\alpha_{3}\}$ where 
$\alpha_{1} =\frac{3}{4} \, \alpha_{\ast}$, $\alpha_{2}=\frac{3}{8} \, \text{min} \{s^{\ast},\gamma_{\ast} \}  $ and $\alpha_{3}=\Gamma_{\ast}.$
\end{proof}

\section{Boundedness \label{sec10}}
In this section we will bound $\displaystyle \sum_{T} A^{T}(u_{\Pi}-u,\delta)$.
Let $\tilde{u}=u_{\Pi}-u$, and $ \delta=u_{h}-u_{I}$.
\begin{lem}
\label{bd1}
Let the virtual element spaces satisfy $(Z1)-(Z3)$ and the control parameter $\delta_{T}$ fulfill the assumption(\ref{dis_coer7}) .Then the following estimation holds
\begin{equation}
\big| \sum_{T}A^{T}(\tilde{u},\delta) \big| \leq C \, h^{k} \, \left(\sum_{T} \eta \, |u|^{2}_{k+1,T} \right)^{1/2} \, ||| \delta ||| \nonumber
\end{equation} 
where $\eta$ is a positive constant.
\end{lem}

\begin{proof}
We will bound $\displaystyle \sum_{T}A^{T}(\tilde{u},\delta)$ term-wise. Using 
Cauchy-Schwarz inequality and polynomial approximation property (\ref{aprox1}) of VE space 
we can bound,  
\begin{eqnarray}
 \sum_{T} \epsilon \int_{T} \nabla \tilde{u} \, \nabla \delta &\leq & \sum_{T} \epsilon \, \| \nabla \tilde{u} \, \|_{0,T} \| \nabla  \delta  \|_{0,T}  \nonumber\\
 & \leq & C \, \epsilon^{1/2} \left(\sum_{T} h^{2k}_{T} \, |u|^{2}_{k+1} \right)^{1/2} \left(\sum_{T} \epsilon^{1/2} \, | \delta |^{2}_{1,T} \right)^{1/2} \nonumber \\
 & \leq & C \, \epsilon^{1/2} \, h^{k} \, \left(\sum_{T} |u|^{2}_{k+1} \right)^{1/2} |||\delta |||
 \label{bd2} 
\end{eqnarray}

Using Greens theorem on each element $T$ we rewrite the bilinear form (\ref{new2}) as,
\begin{eqnarray}
&&\frac{1}{2} \left[<(\vec{b} \cdot \nabla \tilde{u}),\delta >-<(\vec{b} \cdot \nabla \delta),\tilde{u} > -<\text{div} \, b,\tilde{u} \, \delta> \right]\nonumber \\
&& = <(\vec{b} \cdot \nabla \tilde{u}),\delta >-\frac{1}{2}  \int_{\partial T} (\vec{b} \cdot \mathbf{n}) \, \tilde{u} \, \delta \nonumber  \\
&&= -<(\vec{b}.\nabla \delta),\tilde{u} >-<\text{div} \, b,\tilde{u} \, \delta> + \frac{1}{2} \int_{\partial T} (\vec{b} \cdot \mathbf{n}) \, \tilde{u} \, \delta
\label{bd3}
\end{eqnarray}

$\int_{T} c \, \tilde{u} \, \delta$ along with (\ref{bd3}) can be written as,
\begin{equation}
<(c-\text{div} \, \vec{b}),\tilde{u} \, \delta>-<(\vec{b} \cdot \nabla \delta),\tilde{u}>+\frac{1}{2} \int_{T} (\vec{b} \cdot \mathbf{n}) \, \tilde{u} \, \delta 
\label{bd4}
\end{equation}
Using Cauchy-Schwarz inequality and the assumption (\ref{asmp1}) we can estimate

\begin{eqnarray}
\bigg| \sum_{T} <(c-\text{div} \, \vec{b}),\tilde{u} \, \delta>\bigg| &\leq & C \sum_{T} \| \tilde{u} \|_{0,T} \| \, \delta \|_{0,T} \nonumber \\
&\leq & C \left(\sum_{T} h^{2k+2}_{T} \, |u|^{2}_{k+1,T} \right)^{1/2} |||\delta |||
\label{bd5}
\end{eqnarray}
using Cauchy-Schwarz inequality and local polynomial approximation (\ref{aprox1}) and 
interpolation approximation (\ref{aprox2}) of VE space we can estimate second term 
of (\ref{bd4}),

\begin{eqnarray}
\sum_{T} <(\vec{b} \cdot \nabla \delta),\tilde{u} > &\leq & \sum_{T} \delta_{T}^{1/2} \, \|\vec{b} \cdot \nabla \delta \|\  \delta_{T}^{-1/2} \| \tilde{u} \|_{0,T} \nonumber \\
 &\leq & C \left(\sum_{T} \delta^{-1}_{T} \, h^{2k+2}_{T} \, |u|^{2}_{k+1,T} \right)^{1/2} 
  \left(\sum_{T} \delta_{T} \, \|\vec{b} \cdot \nabla \delta \|^{2} \right)^{1/2} \nonumber \\
 &=& C \left(\sum_{T} \left(\frac{h^{2}_{T}}{\delta_{T}} \right) \  h^{2k}_{T} \, |u|^{2}_{k+1,T} \right)^{1/2} |||\delta |||
\label{bd6}
\end{eqnarray}
Again using Cauchy-Schwarz inequality, patch test \cite{irons1972experience} of 
$[|\delta|]$ and polyomial approximation property of VE space we  bound last term of (\ref{bd4}),
  
\begin{eqnarray}
\sum_{T} \int_{\partial T} (\vec{b}.\mathbf{n}) \ \tilde{u} \delta &=& \sum_{e} 
\int_{e} (\vec{b}.\mathbf{n})\tilde{u} [| \delta|] \nonumber \\
  &\leq & C \sum_{T} h^{k+1}_{T} |u|_{k+1,T} \ |\delta|_{1,T} \nonumber \\
   & \leq & C (\sum_{T} (\frac{h^{2}_{T}}{\epsilon}) h^{2k}_{T} |u|^{2}_{k+1})^{1/2} (\sum_{T} \epsilon |\delta|^{2}_{1,T})^{1/2} \nonumber \\
   &=& C h^{k} (\sum_{T} (\frac{h^{2}_{T}}{\epsilon}) |u|^{2}_{k+1})^{1/2} |||\delta |||  
   \label{bd7}
 \end{eqnarray}
Now we bound additional term of bilinear form $A^{T}(\tilde{u},\delta)$.Assuming (\ref{aprox1}) and considering control parameter $\delta_{T}$ satisfies the condition (\ref{dis_coer7}) we can write 
 
 \begin{eqnarray}
 && \sum_{T}  <-\epsilon \Delta \tilde{u}+\vec{b}.\nabla \tilde{u}+c \tilde{u},\delta_{T} \vec{b}.\nabla \delta >_{T}  \nonumber \\
 && \leq \sum_{T} \delta^{1/2}_{T} \|-\epsilon \Delta \tilde{u}+\vec{b}.\nabla \tilde{u} +c \tilde{u} \| \ \delta^{1/2}_{T}  \|\vec{b}.\nabla \delta \| \nonumber \\
 && \leq \sum_{T} \delta_{T}^{1/2} \{ \epsilon \| \Delta \tilde{u} \| +\|\vec{b}.\nabla \tilde{u} \| +\|c \tilde{u} \| \}  \delta^{1/2}_{T} \|\vec{b}.\nabla \delta \| \nonumber \\
 && \leq (\sum_{T} \delta_{T} \{ \epsilon \| \Delta \tilde{u} \|_{0,T}+C h^{k}_{T} |u|_{k+1,T} \}^{2} )^{1/2} \nonumber \\
 && \quad (\sum_{T} \delta_{T} \| \vec{b}.\nabla \delta \|^{2}_{0,T})^{1/2} \nonumber \\
 && C (\sum_{T} (\epsilon+\delta_{T}) h^{2k} |u |^{2}_{k+1,T})^{1/2} ||| \delta |||
 \label{bd8}
 \end{eqnarray}
 
 In last inequality of (\ref{bd8}) we  have used the assumption $\epsilon \delta_{T}<C h^{2}_{T}$, where $C$ is a constant.
 
 Therefore
 \begin{equation}
 | \sum_{T}A^{T}(u_{\Pi}-u,\delta) | \leq C h^{k} (\sum_{T} \eta |u|^{2}_{k+1,T})^{1/2} ||| \delta ||| \nonumber
 \end{equation}
 where $\eta=\epsilon+h^{2}_{T}+(\frac{h^{2}_{T}}{\delta})+\delta_{T}+(\frac{h^{2}_{T}}{\epsilon})$
 \end{proof}
 
 \subsection{boundedness of discrete bilinear form \label{sec11} }
 We will bound $ \sum_{T}A^{T}_{h}(u^{'},\delta) $, where $ u^{'}=u_{I}-u_{\Pi}$ and $\delta=u_{h}-u_{I}$
 
 \begin{lem}
 \label{bond_dis10}
 Let the virtual element spaces satisfy $(Z1)-(Z3)$ and the control parameter $\delta_{T}$ fulfill the assumption(\ref{dis_coer7}) and the bilinear form (\ref{new1}) satisfies stability condition defined in section(\ref{sec7}).Then the following estimation holds
\begin{equation}
| \sum_{T}A^{T}_{h}(u^{'},\delta) | \leq C h^{k} (\sum_{T} \zeta |u|^{2}_{k+1,T})^{1/2} ||| \delta ||| \nonumber
\end{equation} 
where $\zeta $ is a positive constant.
 \end{lem}
 
 \begin{proof}
 Bilinear form $a^{T}_{h}(.,.) $  is symmetric hence it defines an inner-product on $V^{k}_{h} \times V^{k}_{h}$.We can bound inner product by energy norm .
 Using stability property of $a^{T}_{h}(.,.)$ and considering the property (\ref{aprox2})  we can  bound diffusion tern
 
 \begin{eqnarray}
  a^{T}_{h}(u^{'},\delta) &\leq & (a^{T}_{h}(u^{'},u^{'}))^{1/2} (a^{T}_{h}(\delta,\delta))^{1/2} \nonumber \\
 &\leq & \alpha^{\ast} (a^{T}(u^{'},u^{'}))^{1/2} (a^{T}(\delta,\delta))^{1/2} \nonumber \\
 & \leq & \alpha^{\ast} \epsilon^{1/2} \| \nabla u^{'} \|_{0,T}\epsilon^{1/2} \| \nabla \delta \|_{0,T} \nonumber \\
 &\leq & C \alpha^{\ast}\epsilon^{1/2} h^{k}_{T} |u|_{k+1,T} \epsilon^{1/2} \| \nabla \delta \|_{0,T} \nonumber 
\end{eqnarray}  

\begin{equation}
\sum_{T} a^{T}_{h} (u^{'},\delta) \leq C h^{k} (\sum_{T} \epsilon |u|^{2}_{k+1,T})^{1/2} |||\delta |||
\label{bond_dis1}
\end{equation}

similarly 

\begin{eqnarray}
c^{T}_{h} (u^{'},\delta) &\leq & \gamma^{\ast} c_{\text{max}} h^{k+1}_{T} |u|_{k+1} \|\delta \|_{0,T} \nonumber
\end{eqnarray}

\begin{equation}
\sum_{T} c^{T}_{h}(u^{'},\delta) \leq C h^{k}(\sum_{T} h^{2}_{T} |u|^{2}_{k+1} )^{1/2} |||\delta|||
\label{bond_dis2}
\end{equation}
Symmetric part of convection part can be bounded using same idea as earlier
\begin{equation}
b^{T,sym}_{h}(u^{'},\delta) \leq  s^{\ast} \frac{1}{2} |\nabla.\vec{b}| h^{k+1}_{T} |u|_{k+1} \| \delta \|_{0,T} \nonumber
\end{equation}

summing over $T\in \tau_{h}$ and using Cauchy-Schwarz inequality and approximation property (\ref{aprox1}, \ref{aprox2})  we get
\begin{equation}
\sum_{T} b^{T,sym}_{h}(u^{'},\delta) \leq C h^{k}(\sum_{T} h^{2}_{T} |u|^{2}_{k+1} )^{1/2} |||\delta|||
\label{new10}
\end{equation}

Using same technique as diffusion and reaction part we bound additional symmetric stabilization part

\begin{eqnarray}
&&\int_{T} (\vec{b}.\Pi_{k-1}(\nabla u^{'})) \delta_{T} (\vec{b}.\Pi_{k-1}(\nabla \delta)) +s^{T,stab}_{b}((I-\Pi_{k})u^{'},(I-\Pi_{k})\delta) \nonumber \\ 
&& \leq \Gamma^{\ast} \delta^{1/2}_{T} C \|u^{'}\|_{0,T} \delta^{1/2}_{T} \| \vec{b}.\nabla \delta \|_{0,T} \nonumber 
\end{eqnarray}

Summing over all element $T$
\begin{eqnarray}
&&\sum_{T} \{ \int_{T} (\vec{b}.\Pi_{k-1}(\nabla u^{'})) \delta_{T} (\vec{b}.\Pi_{k-1}(\nabla \delta)) +s^{T,stab}_{b}((I-\Pi_{k})u^{'},(I-\Pi_{k})\delta) \} \nonumber \\
&& \leq C h^{k}(\sum_{T} \delta_{T} |u|^{2}_{k+1,T} )^{1/2} |||\delta|||
\label{bond_dis3}
\end{eqnarray}

Using Cauchy-Schwarz inequality and boundedness property of projection operator $\Pi_{k}$ we can estimate 
\begin{eqnarray}
\int_{T} [\vec{b}.\Pi_{k-1}(\nabla u^{'})] \Pi_{k}(\delta) &=& \int_{T} [\vec{b}.\Pi_{k-1}(\nabla u^{'})] \delta \nonumber \\
&\leq & C \|\vec{b} \|_{0,\infty} \| \nabla u^{'} \|_{0,T} \|\delta \|_{0,T} \nonumber \\
&\leq & C \|\vec{b} \|_{0,\infty} h^{k}_{T} |u|_{k+1,T} \|\delta \|_{0,T} \nonumber
\end{eqnarray}
Summing over all element $T\in \tau_{h}$ we get
 
\begin{equation}
\sum_{T} \int_{T} [\vec{b}.\Pi_{k-1}(\nabla u^{'})] \Pi_{k}(\delta) \leq C \|\vec{b} \|_{0,\infty} h^{k} (\sum_{T} |u|^{2}_{k+1} )^{1/2} ||| \delta |||
\label{bond_dis5}
\end{equation}

Using Cauchy-Schwarz inequality and boundedness and orthogonality  property of projection operator  we can estimate
\begin{eqnarray}
\int_{T} [\vec{b}.\Pi_{k-1}(\nabla \delta)] \Pi_{k}(u^{'}) &=& \int_{T} [\vec{b}.\Pi_{k-1}(\nabla \delta)] u^{'} \nonumber \\
&\leq & C h^{k+1}_{T} \|\vec{b} \|_{0,\infty} |u|_{k+1,T} \| \nabla \delta \|_{0,T} \nonumber \\
&\leq & C \|\vec{b}\| \epsilon^{-1/2} h^{k+1}_{T} |u|_{k+1,T} \epsilon^{1/2} |\delta|_{1,T} \nonumber
\end{eqnarray}
 summing over all $T\in \tau_{h}$
\begin{eqnarray}
\sum_{T}\int_{T} [\vec{b}.\Pi_{k-1}(\nabla \delta)] \Pi_{k}(u^{'}) &\leq & \sum_{T} C \|\vec{b}\| \epsilon^{-1/2} h^{k+1}_{T} |u|_{k+1,T} \epsilon^{1/2} |\delta|_{1,T} \nonumber \\
&\leq & C \|\vec{b}\|_{0,\infty} (\sum_{T} (\frac{h^{2}_{T}}{\epsilon}) h^{2k}_{T} |u|^{2}_{k+1} )^{1/2} (\sum_{T} \epsilon |\delta|^{2}_{1,T})^{1/2} \nonumber \\
&\leq & C h^{k} \|\vec{b}\|_{0,\infty}(\sum_{T} (\frac{h^{2}_{T}}{\epsilon})  |u|^{2}_{k+1} )^{1/2} |||\delta |||
\label{bond_dis6} 
\end{eqnarray}

Now we shall bound additional term .We consider the control parameter $\delta_{T}$ satisfies the condition (\ref{dis_coer7}). Using boundedness property of projection operator $\Pi_{k-2}$ we estimate  
\begin{eqnarray}
\int_{T} -\epsilon \Pi_{k-2} (\Delta u^{'}) \delta_{T} (\vec{b}.\Pi_{k-1}(\nabla \delta)) &\leq & \epsilon \delta_{T}^{1/2} \| \Pi_{k-2} \ (\Delta u^{'})\|_{0,T} \  \delta^{1/2}_{T} \| \vec{b}.\nabla \delta \|_{0,T} \nonumber \\
 &\leq &  C \epsilon \delta^{1/2}_{T}  \| \Delta u^{'} \|_{0,T} \  \delta^{1/2}_{T}\| \vec{b}.\nabla \delta \|_{0,T}  \nonumber
\end{eqnarray}

Using Cauchy-Schwarz inequality and using property of the control parameter $\delta_{T}$, $\epsilon \delta_{T} <C h^{2}_{T}$ we estimate
\begin{eqnarray}
&& \sum_{T} \int_{T} -\epsilon \Pi_{k-2} (\Delta u^{'}) \delta_{T} (\vec{b}.\Pi_{k-1}(\nabla \delta)) \nonumber \\ 
&&\leq C (\sum_{T} \epsilon^{2} \delta_{T} \|\Delta u^{'}\|^{2})^{1/2} (\sum_{T} \delta_{T} \|\vec{b}.\nabla \delta\|^{2}_{0,T})^{1/2} \nonumber \\
&& \leq C (\sum_{T} \epsilon h^{2k}_{T} |u|^{2}_{k+1})^{1/2} |||\delta||| \nonumber \\
&& \leq C h^{k} (\sum_{T} \epsilon  |u|^{2}_{k+1})^{1/2} |||\delta|||
\label{bond_dis7}
\end{eqnarray}

Using boundedness property of projection operator $\Pi_{k-1}$ we can write  
\begin{eqnarray}
\int_{T} c \Pi_{k}(u^{'}) \delta_{T} (\vec{b}.\Pi_{k-1}(\nabla \delta))&=& \int_{T} c u^{'} \delta_{T} (\vec{b}.\Pi_{k-1}(\nabla \delta)) \nonumber \\
&\leq & c_{max} \| u^{'}\|_{0,T} \delta_{T} \|\vec{b}.\Pi_{k-1} (\nabla \delta)\|_{0,T} \nonumber \\
&\leq & C \|\vec{b}\|_{0,\infty}  \| \nabla \delta \|_{0,T} \delta_{T} h^{k+1}_{T} |u|_{k+1} \nonumber
\end{eqnarray}
Taking sum over all element $T$ we establish the following inequality
\begin{eqnarray}
&&\sum_{T} c \Pi_{k}(u^{'}) \delta_{T} (\vec{b}.\Pi_{k-1}(\nabla \delta))\nonumber \\
&& \leq C (\sum_{T} \delta^{2}_{T} \epsilon^{-1} h^{2k+2}_{T} |u|^{2}_{k+1,T})^{1/2} (\sum_{T} \epsilon |\delta |^{2}_{1,T} )^{1/2} \nonumber \\
&& \leq C h^{k} (\sum_{T} \delta^{2}_{T} (\frac{h^{2}_{T}}{\epsilon}) |u|^{2}_{k+1,T})^{1/2} |||\delta |||
\label{bond_dis8}
\end{eqnarray}

 Using all above established  inequalities we finally obtain  
\begin{equation}
\bigg| \sum_{T} A^{T}_{h} (u_{I}-u_{\Pi},\delta) \bigg| \leq C h^{k} (\sum_{T} \zeta |u|^{2}_{k+1,T})^{1/2} |||\delta|||
\label{bond_dis9}
\end{equation}

where $\zeta=1+\epsilon+\delta_{T}+h^{2}_{T}+(\frac{h^{2}_{T}}{\epsilon})+\delta^{2}_{T} (\frac{h^{2}_{T}}{\epsilon})$
 
 \end{proof}
\subsection{consistency error estimates \label{sec12}}
Nonconforming VE space $V^{k}_{h} \nsubseteq H^{1}(\Omega)$. 
Function $v_{h}$ on virtual element space $V^{k}_{h}$ is not continuous along interior edges except gauss lobatto points which introduce additional consistency error term$[|v_{h}|]$. We will see  that two consistency error terms will arise in  the proof of convergence analysis in next section.Diffusion part introduce 
$\sum_{T} \int_{\partial T} (\epsilon \nabla u.\mathbf{n}) \delta$ and convection part introduce $\sum_{T} \int_{\partial T} (\vec{b}.\mathbf{n})u \delta$.We use patch test to bound these error terms.Before going into detail proof we introduce some basic result which will help us bound the consistency error terms.

 Let $\mathbf{P}^{e}_{s}:L^{2}(e)\rightarrow \mathbb{P}^{s}(e)$ is the $L^{2}-$ orthogonal projection operator onto the space $\mathbb{P}^{s}(e)$ for $s\geq 0.$
 Let $e\in \varepsilon^{0}_{h}$ be an interior edge and $e$ is shared by two elements $T^{+}$ and $T^{-}$ as a common edge.Standard approximation results \cite{ciarlet1991basic,brenner2008mathematical,di2011mathematical} say that
 \begin{eqnarray}
 \|\nabla u-\mathbf{P}^{e}_{k-1}(\nabla u)\|_{0,e} &\leq & C h^{k-1/2} \|u \|_{k+1,T^{+}\cup T^{-}}  \nonumber \\
 \|[|v_{h}|]-\mathbf{P}^{e}_{0}([|v_{h}|])\|_{0,e} &\leq & C h^{1/2} |v_{h}|_{0,T^{+}\cup T^{-}} 
 \label{conerr1}
\end{eqnarray}    

Let $u\in H^{m}(\Omega),m\geq 3/2$ .$H^{3/2}$ is the space with  minimum regularity to ensure that the analysis can be carried out.
Using patch test of $v_{h}\in V^{k}_{h}$ and stated result(\ref{conerr1}) we can estimate
\begin{eqnarray}
\sum_{T} \int_{\partial T} (\epsilon \nabla u.\mathbf{n}) \delta &=& \sum_{e\in \varepsilon_{h}} \int_{e} \epsilon \nabla u.[|\delta|] \nonumber \\
&\leq & \sum_{e\in \varepsilon_{h}}  \| \nabla u-\mathbf{P}^{e}_{k-1}(\nabla u) \|_{0,e} \ \|[|\delta |]-\mathbf{P}^{e}_{0} ([|\delta|]) \|_{0,e} \nonumber \\
&\leq & C\sum_{T}  h^{k} \epsilon \|u\|_{k+1,T} |\delta |_{1,T} \nonumber \\
&\leq & C h^{k} (\sum_{T} \epsilon \|u\|^{2}_{k+1,T})^{1/2} (\sum_{T} \epsilon |\delta|^{2}_{1,T})^{1/2} \nonumber \\
&\leq & C h^{k}(\sum_{T} \epsilon \|u\|^{2}_{k+1,T})^{1/2} |||\delta |||
\label{conerr2}
\end{eqnarray}
Using same technique as described for estimation (\ref{conerr2}) we can bound 

\begin{eqnarray}
\sum_{T} \int_{\partial T} (\vec{b}.\mathbf{n})u \delta &=& \sum_{e} \int_{e} (\vec{b}.\mathbf{n}) u [|\delta |] \nonumber \\
& \leq & C |\vec{b}.\mathbf{n}|_{\infty} \sum_{e} \| u-\mathbf{P}_{k-1}(u) \|_{0,e} \ \|[|\delta |]-\mathbf{P}_{0}([|\delta|])\|_{0,e}  \nonumber \\
& \leq & C \sum_{T} h^{k+1}_{T} |u|_{k+1,T} |\delta|_{1,T} \nonumber \\
&\leq & C h^{k} (\sum_{T} (\frac{h^{2}_{T}}{\epsilon})|u|^{2}_{k+1,T})^{1/2} (\sum_{T} \epsilon |\delta |^{2}_{1,T})^{1/2} \nonumber\\
&\leq & C h^{k} (\sum_{T} (\frac{h^{2}_{T}}{\epsilon})|u|^{2}_{k+1,T})^{1/2} \ |||\delta |||
\label{conerr3}
\end{eqnarray}

\subsection{Right-hand side estimation \label{sec13}}
In this paragraph we describe  construction of external force term $f$.$H^{1,nc}(\tau_{h};2)$ is the space with minimum patch test to ensure the estimation (\ref{rh1}) $\&$ (\ref{rh2}).Let $\mathbf{P}^{T}_{k}:L^{T}\rightarrow \mathbb{P}^{k}$ be orthogonal $L^{2}$ projection and $(f_{h})|_{T}:={P}^{T}_{k-2}(f)$. We consider locally $f|_{T}\in H^{1}(T)$ .Using $L^{2}$ projection operator ,  Cauchy-Schwarz inequality and standard approximation result we have 
\begin{eqnarray}
\bigg| <f,v_{h}>-<f_{h},v_{h}> \bigg| &=& \bigg| \sum_{T} \int_{T} (f-\mathbf{p}^{T}_{k-2}(f))(v_{h}-\mathbf{p}^{T}_{0}(v_{h}) )\bigg| \nonumber \\
&\leq & \sum_{T} \| f-\mathbf{p}^{T}_{k-2}(f) \|_{0,T} \  \|v_{h}-\mathbf{p}^{T}_{0}(v_{h}) \|_{0,T} \nonumber \\
&\leq & C h^{\text{min}(k-1,s-1)} (\sum_{T} (\frac{h^{2}_{T}}{\epsilon}) |f|^{2}_{s-1,T})^{1/2} (\sum_{T} \epsilon |v_{h}|^{2}_{1,T})^{1/2} \nonumber \\
&=& C h^{\text{min}(k-1,s-1)} (\sum_{T} (\frac{h^{2}_{T}}{\epsilon}) |f|^{2}_{s-1,T})^{1/2} |||v_{h}|||
\label{rh1}
\end{eqnarray}
Using $L^{2}$ projection, Cauchy-Schwarz inequality and standard approximation result we can bound second consistency error term 

\begin{eqnarray}
&&\bigg|<f,\delta_{T} (\vec{b}.\nabla v_{h})>-<f_{h},\delta_{T} (\vec{b}.\nabla v_{h}) > \bigg| \nonumber \\
&& =\bigg| \sum_{T} \int_{T} (f-\mathbf{P}^{T}_{k-2}(f)) \delta_{T} [\vec{b}.\nabla v_{h}-\mathbf{P}^{T}_{0}(\vec{b}.\nabla v_{h})] \bigg|\nonumber \\
&& \leq C h^{\text{min}(k-1,s-1)}(\sum_{T} \delta_{T} |f|^{2}_{s-1})^{1/2} (\sum_{T} \delta_{T} \|\vec{b}.\nabla v_{h}\|^{2}_{0,T})^{1/2} \nonumber \\
&& = C h^{\text{min}(k-1,s-1)}(\sum_{T} \delta_{T} |f|^{2}_{s-1})^{1/2} \ |||v_{h}|||
\label{rh2}
\end{eqnarray}

\textbf{Remark:} In the above estimation (\ref{rh1}) and (\ref{rh2}) constant $C$ depend on $\frac{h^{2}_{T}}{\epsilon}$ and $\delta_{T}$ .In convection dominated case $\epsilon$ is very small quantity and so $\frac{1}{\epsilon}$ is very large quantity.But this does not hamper convergence analysis since as $h_{T}\rightarrow 0$, $h^{2}_{T}$ dominates $\epsilon $.

\section{Convergence Analysis \label{sec14}}

\begin{thm}
\label{cnvr10}
Let $u \in H^{k+1}(\Omega)$ be weak solution of the bilinear form (\ref{temp3}).Let $v_{h}\in V^{k}_{h}$ the be discrete solution of the bilinear form (\ref{temp5}).Let $v_{h}|_{T}\in H^{2}(T)$ , $f|_{T}\in H^{s}(T), (s\geq 1) $ locally and the control parameter $\delta_{T}$ satisfies the condition (\ref{dis_coer7}) .Then the discrete solution satisfy 
\begin{eqnarray}
 |||u-u_{h} ||| &\leq & C h^{k} (\sum_{T} \Gamma |u|^{2}_{k+1,T})^{1/2} +C h^{k} (\sum_{T} \epsilon \|u \|_{k+1,T}^{2})^{1/2} \nonumber \\
 &+& C h^{\text{min}(k-1,s-1)} (\sum_{T} \delta_{T} |f|^{2}_{s-1})^{1/2} \nonumber \\
 &+& C h^{\text{min}(k-1,s-1)} (\sum_{T} (\frac{h^{2}_{T}}{\epsilon}) |f|^{2}_{s-1,T})^{1/2} \nonumber
 \end{eqnarray}
 
 where norm $|||.|||$ is defined in (\ref{dis_coer1}),  and $\Gamma$ is defined by $\Gamma:=\epsilon+\delta_{T}+h^{2}_{T}+(\frac{h^{2}_{T}}{\delta})+(\frac{h^{2}_{T}}{\epsilon})+\delta^{2}_{T}(\frac{h^{2}_{T}}{\epsilon})$. 

\end{thm}

\begin{proof} 
Let $u_{I}$ be interpolation approximation of $ u $ in $ V^{k}_{h} $ .We consider VE space satisfies the estimations(\ref{aprox2}). Introducing $u_{I}$ we divide  $|||u-u_{h}|||$  into two parts as
\begin{eqnarray}
|||u-u_{h}||| &=& |||u-u_{I}+u_{I}-u_{h}||| \nonumber \\
              &\leq & |||u-u_{I}|||+|||u_{I}-u_{h}||| \nonumber 
\end{eqnarray} 
We first bound second term $|||u_{I}-u_{h}|||$  using discrete coercivity of $A_{h}(.,.)$, lemma(\ref{bond_dis10}) and lemma (\ref{bd1}). 
Let us denote
$\delta:=u_{h}-u_{I}$. Using discrete coercivity of $A_{h}(.,.)$ we can write
\begin{eqnarray}
\alpha |||\delta |||^{2} & \leq & A_{h}(\delta,\delta) \nonumber \\
                         & = & A_{h}(u_{h},\delta)-A_{h}(u_{I},\delta) \nonumber \\
                         & =& <f_{h},\delta> + <f_{h},\delta_{T} \vec{b}.\nabla \delta>- A_{h}(u_{I},\delta) \nonumber \\
                         &=& <f_{h},\delta> + <f_{h},\delta_{T} \vec{b}.\nabla \delta> -\sum_{T} A^{T}_{h}(u_{I},\delta)
                         \label{cnvr14} 
\end{eqnarray}
Adding and subtracting $A^{T}_{h}(u_{\Pi},\delta)$ and $A^{T}(u_{\Pi},\delta)$ we get

\begin{eqnarray}
A^{T}_{h}(u_{I},\delta) &= &A^{T}_{h} (u_{I}-u_{\Pi},\delta) +A^{T}_{h}(u_{\Pi},\delta) \nonumber \\
& = & A^{T}_{h} (u_{I}-u_{\Pi},\delta) +A^{T}_{h}(u_{\Pi},\delta) -A^{T}(u_{\Pi},\delta)+A^{T}(u_{\Pi},\delta) \nonumber\\
& =&  A^{T}_{h} (u_{I}-u_{\Pi},\delta) +A^{T}_{h}(u_{\Pi},\delta) -A^{T}(u_{\Pi},\delta) \nonumber \\
& + & A^{T}(u_{\Pi}-u,\delta)+A^{T}(u,\delta)
\label{cnvr11}
\end{eqnarray}

We have shown that $A^{T}_{h}(u_{h},v_{h})$ is polynomial consistent in(), hence $A^{T}_{h}(u_{\Pi},\delta)=A^{T}(u_{\Pi},\delta)$ 

Therefore the estimation (\ref{cnvr11}) reduces 
\begin{eqnarray}
A^{T}_{h}(u_{I},\delta) =A^{T}_{h} (u_{I}-u_{\Pi},\delta) +A^{T}(u_{\Pi}-u,\delta)+A^{T}(u,\delta)
\label{cnvr15}
\end{eqnarray}

We first estimate $\sum_{T}A^{T}(u,\delta)$ .  
Applying Green's theorem on each element $T$  we get 
\begin{eqnarray}
&&\frac{1}{2} (\int_{T} (\vec{b}.\nabla u) \delta -\int_{T}(\vec{b}.\nabla \delta) u-\int_{T} (\nabla.\vec{b})u \delta  ) \nonumber \\
&&= \int_{T} (\vec{b}.\nabla u) \delta  - \frac{1}{2} \int_{\partial T} (\vec{b}.\mathbf{n})u \delta 
\label{cnvr12}
\end{eqnarray}

 summing up the estimation (\ref{cnvr12}) over all element $T\in \tau_{h}$ we get
 additional term $\sum_{T}\frac{1}{2} \int_{\partial T} (\vec{b}.\mathbf{n})u \delta  $ which is described in the following estimation-
\begin{eqnarray}
\sum_{T} A^{T}(u,\delta ) &=& \sum_{T} <(-\epsilon \Delta u+\vec{b}.\nabla u+c u),\delta >_{T} \nonumber \\
& +&\sum_{T} \int_{T} \epsilon \nabla u.\mathbf{n} \delta - \sum_{T} \frac{1}{2} \int_{\partial T} (\vec{b}.\mathbf{n})u \delta \nonumber \\
& +& \sum_{T}<(-\epsilon \Delta u+\vec{b}.\nabla u+c u),\delta_{T} \vec{b}.\nabla \delta >_{T} \nonumber \\
& =& \sum_{T}<f,\delta >_{T}+\sum_{T} <f,\delta_{T}\vec{b}.\nabla \delta > \nonumber \\
&+& \sum_{e} \int_{e} (\epsilon \nabla u.\mathbf{n})[|\delta|]-\frac{1}{2} \sum_{e} \int_{e} (\vec{b}.\mathbf{n}) u [|\delta|]
\label{cnvr13}
\end{eqnarray}

Before putting(\ref{cnvr13}) into (\ref{cnvr14}) we first bound $A^{T}_{h}(u_{I}-u_{\Pi},\delta)$ and $A^{T}(u_{\Pi}-u,\delta)$.

Using lemma(\ref{bond_dis10})  we can write
\begin{equation}
\bigg|\sum_{T} A^{T}_{h}(u_{I}-u_{\Pi},\delta)\bigg| \leq C h^{k} (\sum_{T} \zeta |u|^{2}_{k+1,T})^{1/2} \ |||\delta||| 
\label{cnvr1}
\end{equation}
where $\zeta=\epsilon+\delta_{T}+h^{2}_{T}+(\frac{h^{2}_{T}}{\epsilon})+\delta^{2}_{T}(\frac{h^{2}_{T}}{\epsilon})$

Using lemma(\ref{bd1}) we can write 

\begin{equation}
\bigg|\sum_{T} A^{T}(u_{\Pi}-u,\delta)\bigg| \leq C h^{k} (\sum_{T} \eta |u|^{2}_{k+1,T})^{1/2} ||| \delta |||
\label{cnvr2}
\end{equation}

where $\eta=\epsilon+h^{2}_{T}+\delta_{T}+(\frac{h^{2}_{T}}{\epsilon})+(\frac{h^{2}_{T}}{\delta})$

We have seen that we got two consistency error terms in estimation(\ref{cnvr13}).
Now we shall bound consistency error terms.

Using estimation (\ref{conerr2}) we can write
\begin{equation}
\bigg|\sum_{e\in \varepsilon_{h}} \int_{e} (\epsilon \nabla u.\mathbf{n})[|\delta|] \bigg| \leq C h^{k} (\sum_{T} \epsilon \|u\|^{2}_{k+1,T})^{1/2} ||| \delta|||
\label{cnvr3} 
\end{equation}
Using estimation (\ref{conerr3}) we can write 
\begin{equation}
\bigg| \sum_{T} \int_{\partial T} (\vec{b}.\mathbf{n})u \delta \bigg| \leq C h^{k} (\sum_{T} (\frac{h^{2}_{T}}{\epsilon})) |u|^{2}_{k+1,T})^{1/2} |||\delta|||
\label{cnvr4}
\end{equation}

After putting (\ref{cnvr13}) into (\ref{cnvr14}) we get  two terms $|<f,\delta>-<f_{h},\delta>|$ and  $|\sum_{T}<f,\delta_{T}(\vec{b}.\nabla \delta)>_{T}-<f_{h},\delta_{T}(\vec{b}.\nabla \delta)>_{T}|$.

 Using estimation(\ref{rh1}) we can write 
\begin{equation}
\bigg|<f,\delta>-<f_{h},\delta>\bigg| \leq C h^{\text{min}(k-1,s-1)} (\sum_{T}(\frac{h^{2}_{T}}{\epsilon})|f|^{2}_{s-1,T})^{1/2} |||\delta||| 
\label{cnvr5}
\end{equation}

Using estimation (\ref{rh2}) we can write- 

\begin{eqnarray}
 && \bigg| \sum_{T}<f,\delta_{T}(\vec{b}.\nabla \delta)>_{T}-<f_{h},\delta_{T}(\vec{b}.\nabla \delta)>_{T} \bigg| \nonumber \\
 && C h^{\text{min}(k-1,s-1)} (\sum_{T} \delta_{T} |f|^{2}_{s-1})^{1/2} ||| \delta |||
 \label{cnvr6}
\end{eqnarray}
Putting estimation(\ref{cnvr1}), (\ref{cnvr2}), (\ref{cnvr3}), (\ref{cnvr4}), (\ref{cnvr5}) and (\ref{cnvr6}) in (\ref{cnvr14}) we finally obtain

\begin{eqnarray}
\alpha ||| \delta ||| &\leq & C h^{k} (\sum_{T} \Gamma |u|^{2}_{k+1,T})^{1/2}+Ch^{k} (\sum_{T} \epsilon \|u\|^{2}_{k+1,T})^{1/2} \nonumber \\
&+& C  h^{\text{min}(k-1,s-1)} (\sum_{T} \delta_{T} |f|^{2}_{s-1})^{1/2} \nonumber \\
&+& C  h^{\text{min}(k-1,s-1)} (\sum_{T}(\frac{h^{2}_{T}}{\epsilon})|f|^{2}_{s-1,T})^{1/2} 
\label{cnvr7}
\end{eqnarray}
 where $\Gamma=\epsilon+\delta_{T}+h^{2}_{T}+(\frac{h^{2}_{T}}{\delta})+(\frac{h^{2}_{T}}{\epsilon})+\delta^{2}_{T}(\frac{h^{2}_{T}}{\epsilon}) $
 
Again using triangle inequality we write 
 \begin{eqnarray}
 |||u-u_{h}||| &=&|||u-u_{I}+u_{I}-u_{h} ||| \nonumber \\
               &\leq & |||u-u_{I}|||+|||u_{I}-u_{h}||| \nonumber 
 \end{eqnarray}
 
Using estimation(\ref{aprox2}) we can write 
 \begin{eqnarray}
 |||u-u_{I}||| &=&[\sum_{T} \epsilon |u-u_{I}|^{2}_{1,T}+c_{0} \| u-u_{I}\|^{2}_{0,T}+\sum_{T} \delta_{T} \|\vec{b}.\nabla(u-u_{I})\|^{2}_{0,T}]^{1/2} \nonumber \\
&\leq & C h^{k}[\sum_{T} \epsilon |u|^{2}_{k+1,T}+h^{2}_{T} \sum_{T} |u|^{2}_{k+1,T}+\sum_{T} \delta_{T} |u|^{2}_{k+1,T} ]^{1/2} \nonumber \\
&\leq & C h^{k} (\sum_{T} \Gamma |u|^{2}_{k+1,T})^{1/2} 
  \label{cnvr9}
 \end{eqnarray}
Using inequality (\ref{cnvr9}) we finally obtain required estimation 
 
 \begin{eqnarray}
 |||u-u_{h} ||| &\leq & C h^{k} (\sum_{T} \Gamma |u|^{2}_{k+1,T})^{1/2}+Ch^{k} (\sum_{T} \epsilon \|u\|^{2}_{k+1,T})^{1/2} \nonumber \\
 &+& C h^{\text{min}(k-1,s-1)} (\sum_{T} \delta_{T} |f|^{2}_{s-1})^{1/2} \nonumber \\
 &+& C h^{\text{min}(k-1,s-1)} (\sum_{T} (\frac{h^{2}_{T}}{\epsilon}) |f|^{2}_{s-1,T})^{1/2}
 \end{eqnarray}
\end{proof} 
 
 \subsection{computation issue \label{sec15}}
 The $L^{2}-$ orthogonal projection operator $\Pi^{0}_{k}$ can be computed using elliptic projection operator  $ \Pi^{\nabla}_{k}$

\begin{equation}
\int_{T} p \Pi^{\nabla}_{k} (v_{h}) dT=\int_{T} p v_{h} dT \nonumber
\end{equation}
 
 If degree $(p)< k-2 $  then $ \int_{T} p v_{h} dT$ is computable.
 If degree $(p)=k-1,k $ then  we consider $\int_{T} p v_{h} dT =\int_{T} p \Pi^{\nabla}_{k}(v_{h}) $
 
\textbf{diffusion part:} 
\begin{eqnarray}
\int_{T} \epsilon \Pi_{k-1}(\nabla p) \Pi_{k-1}(\nabla v_{h})&=& \int_{T} \epsilon \nabla p  \Pi_{k-1}(\nabla v_{h}) \nonumber \\
&=&\int_{T} \epsilon \nabla p \nabla v_{h} \nonumber \\
&=& \int_{T} -\epsilon \Delta p v_{h} +\int_{\partial T} \epsilon (\nabla p.\mathbf{n})v_{h}
\label{comp1}
\end{eqnarray}

$\int_{T} \Delta p v_{h}$ is computable using degrees of freedom of $v_{h}$ on triangle since degree$(\Delta p) <k-2$.

$\int_{\partial T} (\nabla p.\mathbf{n})v_{h} $ is computable using degrees of freedom of $v_{h}$ on boundary of triangle since degree$(\nabla p)<k-1$.

\begin{eqnarray}
\int_{T} \vec{b}.\Pi_{k-1}(\nabla p) \Pi_{k}(v_{h}) &=& \int_{T} \vec{b}.\nabla p \Pi_{k}(v_{h}) \nonumber \\
&=& \int_{T} (\vec{b}.\nabla p) v_{h}
\end{eqnarray}

Right hand side  $\int_{T} (\vec{b}.\nabla p) v_{h}$ is computable using elliptic projection operator $\Pi^{\nabla}_{k}$

\begin{eqnarray}
\int_{T} \vec{b}.\Pi_{k-1}(\nabla v_{h}) \Pi_{k}(p) &=& \int_{T} \vec{b}.\Pi_{k-1}(\nabla v_{h}) p \nonumber \\
&=& \int_{T} (\vec{b}. \nabla v_{h}) p +\int_{T} (\vec{b}.\Pi_{k-1}(\nabla v_{h})-\vec{b}.\nabla v_{h})p \nonumber \\
&\approx & \int_{T} (\vec{b}. \nabla v_{h}) p \nonumber \\
&=& -\int_{T} \nabla.(\vec{b}p) v_{h}+\int_{\partial T} (\vec{b}.\mathbf{n}) pv_{h}
\label{comp2}
\end{eqnarray}

error
\begin{equation}
\bigg|\int_{T} (\vec{b}.\Pi_{k-1}(\nabla v_{h})-\vec{b}.\nabla v_{h})p \bigg| \leq  C \| \vec{b} \|_{0,\infty} \|p\|_{0,T} h_{T} |\nabla v_{h}|_{1,T}
\end{equation}

 First part of right hand side $\int_{T} \nabla.(\vec{b}p) v_{h}$  is computable using elliptic projection operator.
 
  For second part we will take $k-1th$ polynomial approximation of $p$, i.e. we will compute $\int_{T} (\vec{b}.\mathbf{n}) \Pi_{k-1}(P)v_{h}$.
  Corresponding error 
  \begin{equation}
  \bigg| \int_{\partial T} (\vec{b}.\mathbf{n}) pv_{h} -\int_{\partial T} (\vec{b}.\mathbf{n}) \Pi_{k-1}(p)v_{h} \bigg| \leq C \|\vec{b}.\mathbf{n} \|_{0,T} h^{1/2}_{T} |p|_{1,T} \|v_{h}\|_{0,T}
  \end{equation}
  
  \begin{eqnarray}
  \int_{T} \vec{b}\Pi_{k}(p) \Pi_{k}(v_{h}) &=&\int_{T} \vec{b} (p) \Pi_{k}(v_{h}) \nonumber \\
  &=& \int_{T} \vec{b} p  v_{h} 
  \end{eqnarray}
  
  Right hand $\int_{T} \vec{b} p  v_{h} $ is computable using elliptic projection operator for  polynomial  $p$ of degrees $k-1,k$.
  
  similarly reaction part $\int_{T} c \Pi_{k}(p) \Pi_{k}(v_{h})$ computable 
  
  \textbf{Stabilization part}
   \begin{eqnarray}
   \int_{T} -\epsilon \Pi_{k-2}(\Delta p) \delta_{T}  \vec{b}.\Pi_{k-1}(\nabla v_{h}) &=& \int_{T} -\epsilon \Delta p \delta_{T} (\vec{b}.\nabla v) \nonumber \\
   &=& \int_{T} \epsilon \delta_{k} \nabla.(\vec{b} \Delta p) v_{h} \nonumber \\
   & -&\int_{\partial T} \epsilon (\vec{b}.\mathbf{n}) \delta_{T} v_{h} \Delta p 
   \end{eqnarray}
   
   Both term of the right hand side is computable.
   
   \begin{eqnarray}
   \int_{T} c \Pi_{k}(p) \delta_{T} \vec{b}.\Pi_{k-1}(\nabla v_{h}) &=& \int_{T} c p \delta_{T} \vec{b}.\Pi_{k-1}(\nabla v_{h})
   \end{eqnarray}
  
  This part is computable using same technique as (\ref{comp2})
  
  \begin{eqnarray}
  \int_{T} \vec{b}.\Pi_{k-1}(\nabla p) \delta_{T} \vec{b}.\Pi_{k-1}(\nabla v_{h})&=& \int_{T} \vec{b}.\nabla p \delta_{T} \vec{b}.\nabla v_{h} \nonumber \\
  &=& \int_{T} -\nabla.(\vec{b}(\vec{b}.\nabla p))v_{h} \nonumber \\ &+&\int_{\partial T} (\vec{b}.\mathbf{n})(\vec{b}.\nabla p) v_{h}
\end{eqnarray}
Both part of right hand side is computable using degrees of freedom of $v_{h}$.

\section{Conclusion \label{sec16}} In this paper we have introduced SDFEM type nonconforming 
VEM framework for convection dominated convection-diffusion reaction equation. To prove 
polynomial consistency we have assumed higher regularity of approximate 
solution, i.e, $v_{h}|_{T} \in H^{2}(T)$ where $v_{h} \in V^{k}_{h}$. The presented 
framework is not convergent in linear nonconforming virtual element space, i.e. it 
requires piecewise higher order polynomial function$(k\geq 2)$ and $f|_{T} \in H^{s}(T)$ 
where $s \geq 1$ and $T$ is an arbitrary element, which may be considered as light 
drawback of this framework. 
\bibliographystyle{plain}
\bibliography{convec_dominatedprob}
\end{document}